\documentclass[a4paper,UKenglish,cleveref, autoref, thm-restate]{lipics-v2021}

\pdfoutput=1 
\hideLIPIcs  

\title{Lions and Contamination: Trees and General Graphs} 

\author{Dohoon Kim}{School of Undergraduate Studies, DGIST, South Korea}{cdjs1432@dgist.ac.kr}{https://orcid.org/0009-0003-7565-9045}{}

\author{Eungyu Woo}{Department of EECS, DGIST, South Korea}{ekwoo@dgist.ac.kr}{https://orcid.org/0009-0000-3621-5175}{}

\author{Donghoon Shin}{Department of EECS, DGIST, South Korea}{dshin@dgist.ac.kr}{https://orcid.org/0009-0004-7088-3846}{}

\authorrunning{D. Kim, E. Woo, and D. Shin} 

\Copyright{Dohoon Kim, Eungyu Woo, and Donghoon Shin} 



\category{} 

\relatedversion{} 

\acknowledgements{We wish to thank our reviewers and colleagues for their insightful feedback.}

\nolinenumbers 

\EventEditors{}
\EventNoEds{0}
\EventLongTitle{34th Annual European Symposium on Algorithms (ESA 2026)}
\EventShortTitle{ESA 2026}
\EventAcronym{ESA}
\EventYear{2026}
\EventDate{August 31-September 4, 2026}
\EventLocation{L'Aquila, Italy}
\EventLogo{}
\SeriesVolume{}
\ArticleNo{}

\begin{document}

\maketitle

\begin{abstract}
This paper investigates a special variant of a pursuit-evasion game called lions and contamination. In a graph where all vertices are initially contaminated, a set of lions traverses the graph, clearing the contamination from every vertex they visit. However, the contamination simultaneously spreads to any adjacent vertex not occupied by a lion. We analyze the relationships among the lion number $\mathcal{L}(G)$, monotone lion number $\mathcal{L}^m(G)$, and the graph's pathwidth $\operatorname{pw}(G)$. Our main results are as follows: 
(a) We prove a monotonicity property: for any graph $G$ and its isometric subgraph $H$, $\mathcal{L}(H)\le \mathcal{L}(G)$.
(b) For trees $T$, we show that the lion number is tightly characterized by pathwidth, satisfying $\operatorname{pw}(T)\le \mathcal{L}(T)\le \operatorname{pw}(T)+1$.
(c) We provide a counterexample showing that the monotonicity property fails for arbitrary subgraphs.
(d) We show that, in contrast to the tree case, pathwidth does not yield a general lower bound on $\mathcal{L}(G)$ for arbitrary graphs.
(e) For any connected graph $G$, we prove the general upper bound $\mathcal{L}(G)\le \operatorname{pw}(G)+1$.
(f) For the monotone variant, we establish the general lower bound $\operatorname{pw}(G)\le \mathcal{L}^m(G)$.
(g) Conversely, we show that $\mathcal{L}^m(G)\le 2\operatorname{pw}(G)+2$ holds for all connected graphs, which is best possible up to a small additive constant.
\end{abstract}

\section{Introduction}
Pursuit-evasion on graphs is often studied through search formulations. In the \emph{lions and contamination} problem, we are given a connected graph \(G=(V_G,E_G)\) and place a set of lions on vertices at time \(t=0\). The occupied vertices are cleared, and all other vertices are contaminated. At each step \(t\ge 1\), every lion either stays or moves to an adjacent vertex, clearing its destination. At the same time, contamination spreads along edges and contaminates any vertex not occupied by a lion that has a contaminated neighbor connected via an edge not used by a lion in that step. The goal is to clear the entire graph. This is an offline pursuit-evasion model, and it differs from both classical graph searching and cop-and-robber games~\cite{Berger09,Dumitrescu07}.

Previous work on lions and contamination has focused mainly on grids and geometric settings. Dumitrescu et al.\ introduced offline variants of the lion-and-man problem and the contamination viewpoint \cite{Dumitrescu07}. Brass et al.\ and Berger et al.\ proved sharp bounds for grid graphs using isoperimetric and separator arguments, including asymptotic bounds in higher-dimensional grids \cite{Berger09,Brass07}. Berger et al.\ also observed that recontamination can be useful, so monotone paths are not automatic in this model \cite{Berger09}. More recently, Bertschinger et al.\ studied monotone and related clearing variants on general graphs and further highlighted that monotonicity is a genuine restriction \cite{Bertschinger23}.

We study the structural theory of lions and contamination on trees and general graphs through its connection to pathwidth. Pathwidth is a standard graph width parameter with central roles in graph searching, layout problems, and structural graph theory \cite{Bienstock91,Kinnersley92,RobertsonI}. Since pathwidth captures the complexity of many search processes, it is natural to ask whether it also controls the number of lions.

We consider the lion number \(\mathcal{L}(G)\) and the monotone lion number \(\mathcal{L}^m(G)\), where monotone means that no cleared vertex is ever recontaminated during the clearing process. Our results are as follows. In \cref{sec:tree}, we prove that \(\mathcal{L}(H)\le \mathcal{L}(G)\) whenever \(H\) is an isometric subgraph of \(G\) and show that every tree \(T\) satisfies \(\operatorname{pw}(T)\le \mathcal{L}(T)\le \operatorname{pw}(T)+1\).
In \cref{sec:subgraph_counterexample}, we give a counterexample demonstrating that the subgraph might require more lions, answering an open question posed in previous work \cite{adams2022lions}.
In \cref{sec:general_graph}, we establish the general upper bound \(\mathcal{L}(G)\le \operatorname{pw}(G)+1\) for every connected graph \(G\), and we demonstrate that pathwidth cannot provide a general lower bound for \(\mathcal{L}(G)\) using our counterexample. As a side result, we also show that our bounds and counterexample can be extended to the zero-visibility cops and robber game, which is detailed in \cref{sec:appendix_zero}. In \cref{sec:monotone_clearing}, we establish upper and lower bounds relating \(\mathcal{L}^m(G)\) and \(\operatorname{pw}(G)\). 
We show that these bounds are best possible up to a small constant, establishing a constant-factor equivalence between $\mathcal{L}^m(G)$ and $\operatorname{pw}(G)$ on connected graphs.

Together, these results give a pathwidth-based framework for lions and contamination on general graphs. Pathwidth gives a tight characterization on trees, a universal upper bound in general, and a constant-factor characterization for monotone clearing.

\section{Preliminaries}\label{sec:preliminaries}
For a formal definition of the problem, we first define the open neighborhood and boundary of the vertex set. The open \emph{neighborhood} of a set $S\subseteq V_G$ is defined as $N(S) := \{w \in V_G \setminus S \mid \exists v \in S \text{ such that } \{w, v\} \in E_G\}$. For simplicity, if we use the notation $N(v)$ for a single vertex $v$, $N(v) := \{w \in V_G \setminus \{v\} \mid \{w, v\} \in E_G\}$. The \emph{boundary} of a set $S\subseteq V_G$ is defined as $\partial(S) := \{w \in S \mid \exists v \in N(S) \text{ such that } \{w, v\} \in E_G\}$. We define $C_t$, $W_t$, and $L_t$ as the sets of vertices that are cleared, contaminated, and occupied by lions at time step $t$, respectively. We define a \emph{path} $\pi=\{\pi_1, \dots\}$ to represent the set of lions' movements at each time step, where the tuple $(u, v)\in \pi_t$ denotes a lion moving from vertex $u$ to $v$ at time step $t$. With these notations, we can formalize the rules of the lions and contamination as follows:
\begin{enumerate}[(i)]
    \item \(W_t = (W_{t-1} \setminus L_t) \cup \{v \in V_G \setminus L_t \mid \exists w \in W_{t-1} \cap N(v) \text{ such that } (v, w) \notin \pi_t \land (w, v) \notin \pi_t \}\)
    \item \(C_t = V_G \setminus W_t\)
\end{enumerate}
Here, $W_0=V_G\setminus L_0$. Additionally, \(\mathcal{L}_\pi(G)\) denotes the number of lions used by the path \(\pi\). Then, we say a graph $G$ is clearable with $\mathcal{L}_\pi(G)$ lions if there exists a path $\pi$ that utilizes $\mathcal{L}_\pi(G)$ lions and $W_T=\emptyset$ for some time step $T$.
Our goal is to determine the minimum number of lions required to clear \(G\), that is,
\[
\mathcal{L}(G) \;:=\; \min_{\pi} \mathcal{L}_{\pi}(G),
\]
where the minimum ranges over all feasible paths \(\pi\) that clear \(G\).

In this paper, we use a standard parameter called pathwidth \cite{dereniowski2012pathwidth,RobertsonI} of a graph $G$ to bound $\mathcal{L}(G)$. The following proposition shows the conditions of pathwidth.

\begin{proposition}\label{prop:pw-basics}
Let $\mathcal{B}=(\mathcal{B}_1, \dots, \mathcal{B}_n)$ be a path decomposition of a graph $G$ where $\mathcal{B}_i\subseteq V_G$ for $i\in \{1, \dots, n\}$.
\begin{enumerate}[(i)]
  \item\label{prop:pw-basics:i} $\bigcup_{i=1}^n \mathcal{B}_i = V_G$,
  \item\label{prop:pw-basics:ii} For $i < j$, if $v \in \mathcal{B}_i \cap \mathcal{B}_j$, then $v \in \mathcal{B}_m$ for every integer $i \leq m \leq j$,
  \item\label{prop:pw-basics:iii} For every edge $\{v, u\} \in E_G$, there exists $i\in\{1,\dots,n\}$ such that $v \in \mathcal{B}_i$ and $u \in \mathcal{B}_i$.
\end{enumerate}
\end{proposition}
The width of the path decomposition $\mathcal{B}$ is $\max_{i=1,\dots, n}|\mathcal{B}_{i}|~-~1$, and the \emph{pathwidth} of a graph $G$, $\operatorname{pw}(G)$, is the minimum width over all path decompositions of $G$.

To facilitate our proofs for the lower and upper bounds, we define two auxiliary operations, which we call \emph{remote operations}. Although the standard game rules do not allow clearing a vertex without a lion nor contaminating a vertex without neighboring contamination, defining these virtual operations allows us to analyze the game state more rigorously. We define \emph{remote clear} and \emph{remote contamination} as follows:
\begin{definition}[Remote contamination]
    A \emph{remote contamination} at vertex $v$ at time step $t$ is the operation of adding $v$ to $W_t$, regardless of whether $v$ has any adjacent contaminated vertex.
\end{definition}
\begin{definition}[Remote clear]
    A \emph{remote clear} at vertex $v$ at time step $t$ is the operation of adding $v$ to $C_t$, regardless of whether any lion occupies $v$.
\end{definition}

For simplicity of analysis, we record these hypothetical operations in our path $\pi$. We emphasize that these are strictly analytical tools; thus, we assume that any valid clearing path $\pi$ does not contain these remote operations unless explicitly stated otherwise. We first prove the following supporting lemmas:

\begin{lemma}\label{lemma:contamination_subset_lemma}
Let $W_t$ and $W'_t$ be the contaminated sets under paths $\pi$ and $\pi'$, respectively. If $W'_t \subseteq W_t$ and $\pi_{t+k} = \pi'_{t+k}$ for every integer $k \ge 1$, then $W'_{t+k} \subseteq W_{t+k}$.
\end{lemma}

\begin{proof}
We prove by induction on $k$. For the base case $k=1$, by the problem definition, the contaminated set at time $t+1$ under path $\pi$ is given by $W_{t+1} = (W_t \setminus L_{t+1}) \cup \{v \in V_G \setminus L_{t+1} \mid \exists w \in W_t \cap N(v) \text{ such that } (v, w) \notin \pi_t \land (w, v) \notin \pi_t\}$.
Since $W'_t \subseteq W_t$, it is clear that $(W'_t \setminus L_{t+1}) \subseteq (W_t \setminus L_{t+1})$. 
Furthermore, since $\pi$ and $\pi'$ operate on the same vertex set $V_G$ and $\pi_{t+1} = \pi'_{t+1}$, the set of guarded edges is identical for both paths. Because $\{w \in W'_t \cap N(v)\} \subseteq \{w \in W_t \cap N(v)\}$, any vertex newly contaminated from $W'_t$ will also be contaminated from $W_t$. Thus, the set of newly contaminated vertices under $\pi'$ is a subset of the corresponding set under $\pi$.

Combining these two parts yields $W'_{t+1} \subseteq W_{t+1}$. By applying this argument inductively for any $k \ge 1$, assuming $W'_{t+k-1} \subseteq W_{t+k-1}$ and given $\pi_{t+k} = \pi'_{t+k}$, we obtain $W'_{t+k} \subseteq W_{t+k}$ for every integer $k \ge 1$.
\end{proof}

\begin{lemma}\label{lemma:remote_clear_lemma}
Let $\pi$ be a path that clears the graph $G$, and let $\pi'$ be a path obtained by adding some remote clear operations to $\pi$. Then, $\pi'$ can also clear the graph $G$.
\end{lemma}

\begin{proof}
Let $W_t$ and $W'_t$ be the contaminated sets under paths $\pi$ and $\pi'$, respectively. We first consider the case where $\pi'$ is obtained by adding a remote clear operation to a vertex $v$ at time step $t$. Under this operation, for all $t' \le t$, we have $W'_{t'} \subseteq W_{t'}$ (specifically, $W'_t = W_t \setminus \{v\}$ if $v$ was cleared, or $W'_t = W_t$ otherwise). 
Since no other operations are changed, $\pi'_{t+k} = \pi_{t+k}$ for all $k \ge 1$. By \cref{lemma:contamination_subset_lemma}, it follows that $W'_{t+k} \subseteq W_{t+k}$ for all $k \ge 1$. Thus, $W'_{t'} \subseteq W_{t'}$ holds for all time steps $t' \ge 0$.
Because $\pi$ clears the graph $G$, there exists some time step $T$ such that $W_T = \emptyset$. Since $W'_T \subseteq W_T$, it follows that $W'_T = \emptyset$, meaning $\pi'$ also clears $G$.
If $\pi'$ contains multiple remote clear operations, we can apply this logic inductively on the number of operations. Repeatedly adding single remote clear operations preserves the subset property globally, thus $\pi'$ can always clear the graph $G$.
\end{proof}

\begin{lemma}\label{lemma:remote_contaminate_lemma}
Let $\pi'$ be a path obtained by adding some remote contamination operations to $\pi$. If $\pi'$ clears the graph $G$, then $\pi$ can also clear the graph $G$.
\end{lemma}

\begin{proof}
The proof is analogous to \cref{lemma:remote_clear_lemma}. Adding a remote contamination operation to $\pi$ to form $\pi'$ implies that at the time $t$ of the operation, the contaminated set under $\pi$ is a subset of the contaminated set under $\pi'$ (i.e., $W_t \subseteq W'_t$). 
By applying \cref{lemma:contamination_subset_lemma}, this implies $W_{t'} \subseteq W'_{t'}$ for all time steps $t' \ge 0$. Since $\pi'$ clears the graph, there exists a time step $T$ where $W'_T = \emptyset$. Therefore, $W_T \subseteq W'_T = \emptyset$, which means $W_T = \emptyset$. Thus, $\pi$ also clears the graph $G$.
\end{proof}

\section{Isometric Subgraphs and Trees}\label{sec:tree}
In this section, we show that a supergraph always requires at least as many lions as its isometric subgraph, and prove that the required number of lions to clear a tree can be completely characterized in terms of pathwidth.

\begin{lemma}\label{lemma:constant_remote_lemma}
Given a graph $G$ and its subgraph $H$, let $S=V_G\setminus V_H$. Let $\pi'$ be the path obtained by adding a remote clear operation to $\pi$ for every vertex in $S$ at every time step. If $\pi$ does not involve any vertices in $S$, then $\pi$ clears $H$ if and only if $\pi'$ clears $G$, and in either case $\mathcal{L}_\pi(H) = \mathcal{L}_{\pi'}(G)$.
\end{lemma}

\begin{proof}
Let $W_t$ be the set of contaminated vertices in $H$ under path $\pi$ at time step $t$, and let $W'_t$ be the set of contaminated vertices in $G$ under path $\pi'$ at time step $t$. Note that $\pi$ and $\pi'$ share the exact same lion trajectories ($L_t$) within $V_H$, and $\pi$ has no lions in $S$.

Since path $\pi'$ applies remote clear on $S$ at every time step, every vertex in $S$ remains cleared for all $t$. That is, $S \cap W'_t = \emptyset$ for all $t$. Consequently, no contamination can spread from $S$ to $H$ in graph $G$. This mimics the boundary condition of an isolated graph $H$, where no external contamination exists.

Thus, for any vertex $v \in V_H$, the condition for contamination (neighbors being contaminated and no lion blocking) is identical in both scenarios. By induction, $W_t = W'_t \cap V_H$ for all $t$. Since we established that $W'_t \cap S = \emptyset$, it follows that $W'_t = W_t$.

Since both paths utilize the same set of lions, the number of lions required is identical. Therefore, $\pi$ successfully clears $H$ (i.e., $W_t = \emptyset$) if and only if $\pi'$ successfully clears $G$ (i.e., $W'_t = \emptyset$). Consequently, if $\pi$ is a valid clearing path for $H$, its extension $\pi'$ clears $G$ using the exact same number of lions, thus $\mathcal{L}_\pi(H)=\mathcal{L}_{\pi'}(G)$.
\end{proof}

\begin{theorem}\label{theorem:isometric_subgraph_theorem}
For a graph $G$ and its isometric subgraph $H$, $\mathcal{L}(H) \leq \mathcal{L}(G)$.
\end{theorem}

\begin{proof}
Let $S=V_G\setminus V_H$ and let $\pi$ be a path that clears $G$ such that $\mathcal{L}_\pi(G)=\mathcal{L}(G)$. By \cref{lemma:remote_clear_lemma}, if we denote $\pi'$ as the path after applying remote clear operations to every vertex in $S$ at every time step $t$, then $\mathcal{L}_{\pi'}(G) \leq \mathcal{L}(G)$. Now, consider the trajectory of a single lion in $\pi'$. If there exists any trajectory segment $(v, u_1, \dots, u_n, w)$ where $u_i \in S$, we can replace it with $(v, u'_1, \dots, u'_n, w)$ where $u'_i \notin S$ (i.e., $u'_i \in V_H$). Since $H$ is an isometric subgraph of $G$, we have $d_H(v, w) = d_G(v, w) \le n+1$. Because the game rules allow a lion to stay on the same vertex, we can construct a valid trajectory $(v, u'_1, \dots, u'_n, w)$ entirely within $H$ by taking a shortest path and adding waiting steps if necessary. Since we consistently apply the remote clear operation to $u_i$, $u_i$ remains cleared regardless of the lion's presence. This change of trajectory never decreases the size of $C'_t$. Thus, if $\pi''$ is the path after modifying every such trajectory, then $\mathcal{L}_{\pi''}(G) \leq \mathcal{L}_{\pi'}(G)$. Finally, by \cref{lemma:constant_remote_lemma}, after dropping the remote clear operations from $\pi''$ to create $\pi'''$, we get $\mathcal{L}_{\pi'''}(H) = \mathcal{L}_{\pi''}(G) \leq \mathcal{L}(G)$ for every path $\pi$. Since $\pi'''$ does not have any remote operations in it, $\mathcal{L}(H)\leq \mathcal{L}_{\pi'''}(H)\leq \mathcal{L}(G)$ holds.
\end{proof}

Since a tree and its subtrees form isometric subgraphs, we can use \cref{lemma:constant_remote_lemma} and \cref{theorem:isometric_subgraph_theorem} to prove the lower and upper bounds of required lions on trees.
\begin{lemma}\label{lemma:neighboring_contamination_component_lemma} For every connected component $S\subseteq W_t$ where $N(S)\subseteq C_t$, $N(S)\cap L_t\neq \emptyset$.
\end{lemma}

\begin{proof}
At time step $0$, it is clear by definition that every contaminated component has at least one neighboring lion. To prove by contradiction, let $S$ be the \emph{first} contaminated component satisfying the condition, i.e., $S \subseteq W_t$, $N(S)\subseteq C_t$, and $N(S)\cap L_t=\emptyset$, and no other vertex set $S'$ satisfies the condition at any time step $t'<t$.

Suppose there is a vertex $w\in \partial(S)\cap W_{t-1}$. By the problem definition, contamination at $w$ spreads to $N(w)$ if there are no lions in $N(w)$ at time step $t$ and no lions traverse any edge between $w$ and a vertex in $N(w)$ at time step $t$. However, since $S \subseteq W_t$ directly implies $S\cap L_t=\emptyset$, and given $N(S)\cap L_t=\emptyset$, there are no lions on $w$ nor moves along the edge between $w$ and $N(w)$ at time step $t$. This implies that $N(S)\nsubseteq C_t$, which is a contradiction. Thus, $\partial(S)\subseteq C_{t-1}$ must hold.

Suppose that $S\setminus \partial(S)\subseteq C_{t-1}$. Since $\partial(S)\subseteq C_{t-1}$, it follows that $S \subseteq C_{t-1}$. By the condition $S\subseteq W_t$, $N(S)\cap W_{t-1}\neq \emptyset$ must hold. However, the condition $N(S)\cap L_t=\emptyset$ implies that any contamination on $N(S)$ can never be cleared at time step $t$; thus, $N(S)\nsubseteq C_{t}$, which is a contradiction. Therefore, $S\setminus \partial(S)\nsubseteq C_{t-1}$, which implies $(S\setminus \partial(S))\cap W_{t-1}\neq \emptyset$. We can find $S'\subset S$ such that $S'\subseteq W_{t-1}$ and $N(S')\subseteq C_{t-1}$. Furthermore, it naturally follows that $N(S')\cap L_{t-1}= \emptyset$ since $(S\cup N(S))\cap L_t=\emptyset$. This contradicts the fact that $t$ is the \emph{first} time step satisfying the condition.
\end{proof}

\begin{lemma}\label{lemma:subtree_on-graph_lion_lemma}
Given a graph $G$ and its isometric subgraph $H=(V_H, E_H)$ where $|\partial(V_H)|=1$, there exists at least one time step where at least $\mathcal{L}(H)$ lions are present in $V_H$ to clear $G$.
\end{lemma}

\begin{proof}
Suppose we have a path $\pi$ that clears $G$ without ever having $K=\mathcal{L}(H)$ lions on $V_H$. Let vertex $v$ be the only vertex in $\partial(V_H)$. Now, let us apply the remote clear operation on $V_G\setminus V_H$ at every time step $t$. By \cref{lemma:remote_clear_lemma}, we can guarantee that this path can also clear $G$. Next, we modify every lion leaving $V_H$ through $v$ to stay at $v$, and remove any excess lions on $v$ so that the total number of lions on the path $\pi$ remains at most $K-1$. Since $\pi$ never allows more than $K-1$ lions on $V_H$, there must exist enough lions to remove at the vertex $v$. Since $V_G\setminus V_H$ always remains a set of cleared vertices, this modified path can also clear the graph $G$. Finally, by \cref{lemma:constant_remote_lemma}, the path $\pi'$ that contains the modified trajectory of lions without remote clear operations must clear the graph $H$, which is a contradiction to $\mathcal{L}(H)=K$.
\end{proof}

\begin{lemma}\label{lemma:instant_contamination_lemma}
Let $\pi$ be a path that clears the graph $G$. For any time step $t$, if a completely contaminated separating set $S \subseteq W_t$ divides graph $G$ and one of its resulting connected components $H$ does not contain any lions, modifying the path $\pi$ to $\pi'$ by adding a remote contamination to every vertex in $H$ at time step $t$ also clears $G$.
\end{lemma}

\begin{proof}
Let $W_t$ and $W'_t$ be the contaminated sets under paths $\pi$ and $\pi'$ at time step $t$, respectively. Let $D_i$ denote the set of vertices in $V_H$ whose shortest distance from $S$ is $i$, $D_0$ denote $S$, and let $D^*$ denote the set of all vertices in $V_G \setminus (V_H \cup S)$. By construction, the contamination status of every vertex in $D^*$ and $D_0$ is identical under both paths $\pi$ and $\pi'$ at time step $t$ (i.e., $W_t \cap (D^* \cup D_0) = W'_t \cap (D^* \cup D_0)$).

Assume inductively that the contamination status of vertices in $D^*, D_0, \dots, D_n$ remains identical under both paths at time step $t+n$, and that $D_n$ is fully contaminated ($D_n \subseteq W_{t+n}$). Since all lions are initially outside $H$ and cannot be on the fully contaminated set $S$, the distance between $D_{n+1}$ and any lion at time $t+n$ is strictly greater than 1. Thus, the contamination spreads from $D_n$, and every vertex in $D_{n+1}$ will become contaminated at time step $t+n+1$. Thus, $W_{t+n+1} \cap D_{n+1} = W'_{t+n+1} \cap D_{n+1} = D_{n+1}$, making their contamination status identical.

Furthermore, since every vertex in $D_n$ is fully contaminated at time step $t+n$, any contamination spreading from $N(D_n)=D_{n-1}\cup D_{n+1}$ to $D_n$ does not alter the set of contaminated vertices in $D_n$; instead, only lions moving into $D_n$ can change vertices from contaminated to cleared. Since the lion movements in $\pi$ and $\pi'$ are identical, the set of contaminated vertices in $D_n$ remains identical under both paths at time step $t+n+1$ (i.e., $W_{t+n+1} \cap D_n = W'_{t+n+1} \cap D_n$).

By the definition of the problem, the cleared and contaminated sets are affected only by lion movements and the contamination status of their closed neighborhoods. Since the sets of cleared and contaminated vertices from $N(D^*)$ through $N(D_{n-1})$ are identical at time step $t+n$ and the lion trajectories are exactly the same, the contamination status of vertices in $D^*$ through $D_{n-1}$ also remains identical at time step $t+n+1$. It follows that $D^*, D_0, D_1, \dots , D_{n+1}$ remain identical under both paths at time step $t+n+1$.

Thus, by mathematical induction, we show that after a sufficient number of time steps, the exact partition of $V_G$ into cleared and contaminated sets will become identical under both paths. Since $\pi$ clears $G$, $\pi'$ must also clear $G$.
\end{proof}

\begin{lemma}\label{lemma:lion_lower_bound_lemma}
Given a tree $T$ and an integer $k \ge 1$, $\mathcal{L}(T) \ge k+1$ if and only if there exists a vertex $v$ such that $T - v$ has at least three connected components $T_1, T_2, T_3$ satisfying $\mathcal{L}(T_i) \ge k$ for $i\in\{1, 2, 3\}$.
\end{lemma}

\begin{proof}
By \cref{theorem:isometric_subgraph_theorem}, if there exists at least one subtree $T'$ satisfying $\mathcal{L}(T') \ge k+1$, then $\mathcal{L}(T) \ge k+1$ holds. Thus, for the remainder of the proof, we assume that for every connected component $T'$ where $\mathcal{L}(T') \ge k$, we have exactly $\mathcal{L}(T') = k$.

\proofsubparagraph*{Forward Direction ($\mathcal{L}(T) \le k$ when the condition fails)}
Since the condition fails, for every vertex $v$, the induced subgraph $T - v$ has at most two connected components requiring $k$ lions. Now, we choose a vertex $v$ in $T$ to maximize $\mathcal{L}(T_2)$, where $T_1, T_2, \dots$ are the connected components of $T - v$ ordered such that $\mathcal{L}(T_1)\geq \mathcal{L}(T_2) \geq \dots$. If $\mathcal{L}(T_1)<k$, then it immediately follows that $\mathcal{L}(T)\leq k$ since placing a lion at $v$ and clearing every other component with $k-1$ lions sequentially would clear $T$. Thus, we can assume $\mathcal{L}(T_1)=k$.

Now, let $r_1$ be the root of $T_1$, and let $T'$ be a component that satisfies $T'\subseteq T_1 - r_1$ and $\mathcal{L}(T')=k$. If there is no such $T'$, then it is straightforward to clear $T_1$ by placing a lion on $r_1$ and using $k-1$ lions on each component in $T_1- r_1$. Furthermore, it is evident that no other component in $T_1$ requires $k$ lions since we already have $\mathcal{L}(T_2)$ maximized and at most two connected components require $k$ lions. By recursively applying this procedure, since $T$ is a finite tree, we can obtain $r_1'$ such that no subtree rooted at the children of $r_1'$ requires $k$ lions, while the subtree rooted at $r_1'$ requires $k$ lions. We apply the following procedure: starting with a lion on $r_1'$, we use $k-1$ lions to clear its subtrees. We then iteratively move the lion to its parent, clearing the adjacent contaminated subtrees with the remaining $k-1$ lions at each step, proceeding upward until $T_1$ is entirely cleared. Finally, moving every lion to $v$ clears $T - T_2$.

If $\mathcal{L}(T_2) < k$, we can simply place one lion at $v$ and use the remaining $k-1$ lions to entirely clear $T_2$. Now, assume $\mathcal{L}(T_2) = k$. Let $r_2$ be the root of $T_2$. We begin by placing a lion at $r_2$ and sequentially clearing every adjacent subtree that requires strictly fewer than $k$ lions. Note that the lion stationed at $r_2$ acts as a blocker, preventing contamination from spreading between the subtrees during this process. Furthermore, by our given condition, there exists at most one subtree $T'$ rooted at a child of $r_2$ satisfying $\mathcal{L}(T') = k$. If such a subtree $T'$ exists, we move the lion to its root and recurse. Otherwise, the procedure terminates and $T_2$ is completely cleared. Since no recontamination occurs in $T - T_2$ throughout this process, the entire tree $T$ is cleared, which implies $\mathcal{L}(T)\leq k$.

\proofsubparagraph*{Backward Direction ($\mathcal{L}(T) \ge k+1$ when the condition holds)}
Since $|\partial(T_i)|=1$ and $T_i$ is an isometric subgraph of $T$ for each component $i \in \{1, 2, 3\}$, we can apply \cref{lemma:subtree_on-graph_lion_lemma}; thus, there must be at least one time step where all $k$ lions are present entirely within $T_i$. Let $t_i$ be the \emph{last} time step during which all $k$ lions are located within $T_i$. Since we only have $k$ lions, these time steps must be distinct. Without loss of generality, assume $t_1 < t_2 < t_3$.
At time step $t_2$, since $t_3 > t_2$, there is still a contaminated component in $T - T_2$. Because all $k$ lions are entirely within $T_2$ at $t_2$, \cref{lemma:neighboring_contamination_component_lemma} requires this contamination to be adjacent to a lion in $T_2$. Thus, $N(T_2)$ must be contaminated.
By \cref{lemma:instant_contamination_lemma}, since $N(T_2)$ acts as a contaminated separating set, we can remotely contaminate every vertex of $T - T_2$, which includes $T_1$, at time step $t_2$, and the path must still successfully clear the graph. This implies that the path must clear $T_1$ from a fully contaminated state \emph{after} time $t_2$, which requires placing all $k$ lions in $T_1$ at some time $t > t_2$. This directly contradicts the definition of $t_1$ being the \emph{last} time step for $T_1$. Therefore, $k$ lions are insufficient, and $\mathcal{L}(T) \ge k+1$.
\end{proof}

Then, we use following lemma from \cite{takahashi1994minimal} to prove the main theorem:

\begin{lemma}[Lemma 2.10 of \cite{takahashi1994minimal}]\label{lemma:pathwidth_lemma}
For any tree $T$ and integer $k \geq 1$, $\operatorname{pw}(T) \geq k+1$ if and only if $T$ has a vertex $v$ such that $T - v$ has at least three connected components with pathwidth $k$ or more.
\end{lemma}

\begin{theorem}\label{theorem:pathwidth_theorem}
For a tree $T$, $\operatorname{pw}(T) \leq \mathcal{L}(T) \leq \operatorname{pw}(T)+1$.
\end{theorem}
\begin{proof}
We proceed by strong induction on the number of vertices $|V_T|=n$ in the tree $T$. For the base case of $n=1$, we inherently have $\operatorname{pw}(T)=0$ and $\mathcal{L}(T)=1$, thus $\operatorname{pw}(T)\leq \mathcal{L}(T)\leq \operatorname{pw}(T)+1$ holds. Assume as the inductive hypothesis that $\operatorname{pw}(T') \leq \mathcal{L}(T') \leq \operatorname{pw}(T')+1$ holds for every tree $T'$ with $|V(T')|\leq n$. Now, consider a tree $T$ with $|V_T|=n+1$ and let $\mathcal{L}(T)=k$.

First, since $\mathcal{L}(T)=k < k+1$, the contrapositive of \cref{lemma:lion_lower_bound_lemma} states that for every vertex $u \in V_T$, $T - u$ has at most two connected components with $\mathcal{L}(T_i) \geq k$. Since the inductive hypothesis suggests $\operatorname{pw}(T_i) \leq \mathcal{L}(T_i)$, $T - u$ must also have at most two connected components with $\operatorname{pw}(T_i) \geq k$. By \cref{lemma:pathwidth_lemma}, this implies $\operatorname{pw}(T) \leq k$.

Second, since $\mathcal{L}(T) \geq k$, by \cref{lemma:lion_lower_bound_lemma}, there exists a vertex $v$ such that $T -v$ has at least three connected components $T_1, T_2, T_3$ satisfying $\mathcal{L}(T_i) \geq k-1$. The inductive hypothesis states $\mathcal{L}(T_i) \leq \operatorname{pw}(T_i)+1$, which gives $\operatorname{pw}(T_i) \geq \mathcal{L}(T_i)-1 \geq k-2$. Since $T -v$ has at least three components with $\operatorname{pw}(T_i) \geq k-2$, \cref{lemma:pathwidth_lemma} implies that $\operatorname{pw}(T) \geq (k-2)+1 = k-1$.

Combining these two inequalities gives us $k-1 \leq \operatorname{pw}(T) \leq k$, which exactly means $\operatorname{pw}(T) \leq \mathcal{L}(T) \leq \operatorname{pw}(T)+1$ holds for $T$.
\end{proof}

\begin{remark}
Since it is already proven that the pathwidth of the complete binary tree with depth $h$ (i.e., a complete binary tree with $2^{h+1}-1$ vertices) equals $\lceil \frac{h}{2} \rceil$ in \cite{bodlaender1998partial} and \cref{lemma:lion_lower_bound_lemma} implies that $\mathcal{L}(T_h)=\mathcal{L}(T_{h-2})+1$, the bound $\operatorname{pw}(T)\leq \mathcal{L}(T)\leq \operatorname{pw}(T)+1$ is tight.
\end{remark}

\section{Non-Monotonicity under Subgraph Inclusion}\label{sec:subgraph_counterexample}
While the lion number is monotonic for isometric subgraphs like trees, it is natural to ask if this property holds generally. In this section, we answer this in the negative. We demonstrate that an arbitrary subgraph may require strictly more lions to be cleared than its supergraph by providing a specific recursive construction.

\begin{figure}[t]
    \centering
    \includegraphics[width=\textwidth]{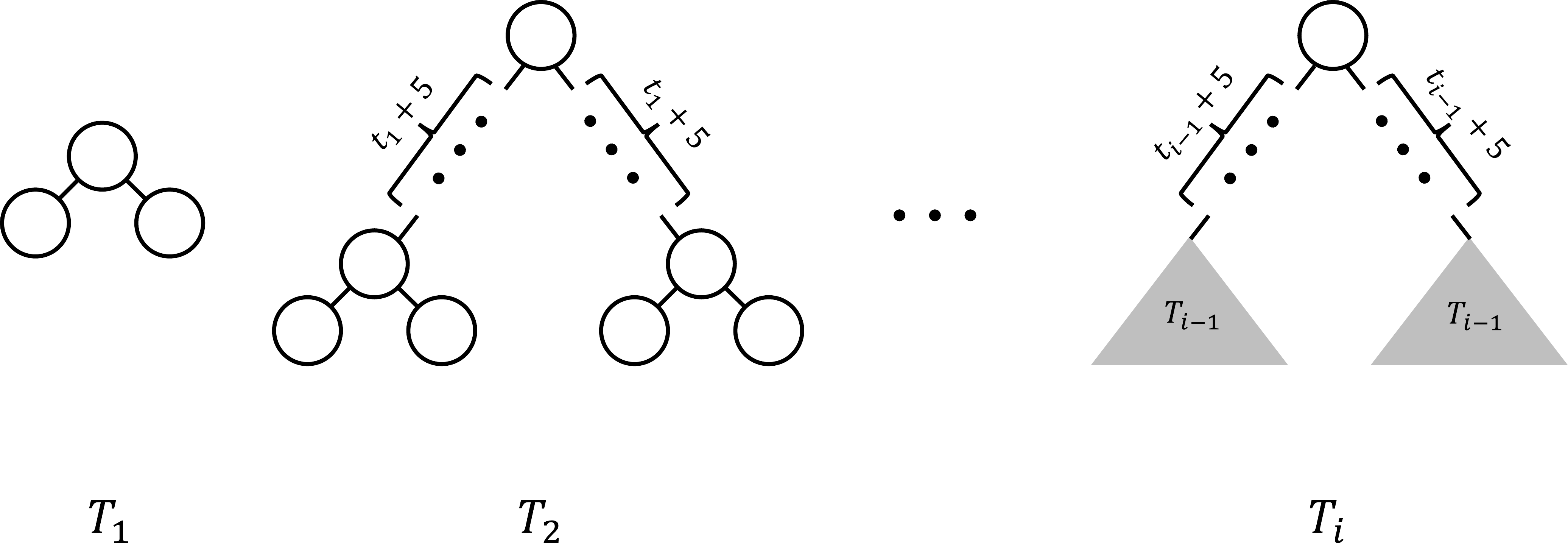}
    \caption{The construction of trees $T_1, T_2, \dots, T_i$.}
    \label{fig:figure}
\end{figure}

\begin{theorem}\label{theorem:subgraph_counterexample}
    The lion number is not monotonic under general subgraph inclusion.
\end{theorem}

\begin{proof}
To demonstrate this, we construct a family of trees $T_i$ and their supergraphs $G_i$ such that the supergraph requires fewer lions than its subgraph. Let $T_1$ be a tree consisting of a root and two leaves, and let $G_1$ be the graph formed by adding a universal vertex $u$ connected to all vertices in $T_1$. We define $\pi^{(1)}$ as the path that clears $G_1$ using exactly $3$ lions. In this strategy, one lion is placed on $u$ and the other two lions occupy the two leaves of $T_1$. In the subsequent time step, the two lions simultaneously slide to the root of $T_1$. This path $\pi^{(1)}$ requires $t_1 = 1$ time step, with the two lions arriving at the root of $T_1$ only at the final step.

We construct $T_{i+1}$ and $G_{i+1}$ inductively by assuming that a path $\pi^{(i)}$ clears $G_i$ using $3$ lions in $t_i$ time steps. We further assume that $\pi^{(i)}$ maintains one lion on the universal vertex $u$ throughout the process, while the other two lions arrive at the root of $T_i$ simultaneously at the last time step. The tree $T_{i+1}$ is constructed with a new root vertex $r_{i+1}$ connected to the roots of two disjoint copies of $T_i$, referred to as $T_L$ and $T_R$, via paths containing $t_i+5$ internal vertices. Consequently, the distance from $r_{i+1}$ to the root of $T_L$ is exactly $t_i+6$ edges. Let $G_{i+1}$ be the supergraph formed by adding a universal vertex $u$ connected to every vertex in $T_{i+1}$. A new path $\pi^{(i+1)}$ clears $G_{i+1}$ using $3$ lions according to the following steps:

\begin{enumerate}
    \item One lion remains at the universal vertex $u$ for the entire duration of the clearing process.
    \item The remaining two lions clear the left subtree $T_L$ by mimicking $\pi^{(i)}$.
    \item The two lions move up the path from the root of $T_L$ until they reach the child of $r_{i+1}$.
    \item Both lions jump to the universal vertex $u$, leaving the left path exposed to the contaminated root $r_{i+1}$.
    \item The lions jump from $u$ to the starting positions of $\pi^{(i)}$ within the right subtree $T_R$.
    \item The path $\pi^{(i)}$ is executed to clear $T_R$ in $t_i$ time steps, concluding with both lions at the root of $T_R$.
    \item One lion moves from the root of $T_R$ to the root of $T_L$ via $u$.
    \item The lions at the roots of $T_L$ and $T_R$ simultaneously slide up their respective paths, arriving at and clearing $r_{i+1}$ at the final time step of the process.
\end{enumerate}

It remains to verify that $T_L$ is not recontaminated during Steps 4 through 7. The total time elapsed between the lions vacating the left path in Step 4 and a lion returning to the root of $T_L$ in Step 7 is exactly $1 + 1 + t_i + 2 = t_i + 4$ time steps. Given that contamination spreads at a rate of one edge per time step, it advances at most $t_i+4$ edges down the left path from $r_{i+1}$. Since the total distance to the root of $T_L$ is $t_i+6$ edges, the contamination cannot reach $T_L$, ensuring it remains cleared. Furthermore, $\pi^{(i+1)}$ preserves the inductive condition where the lions arrive at the root $r_{i+1}$ only at the last time step.

Through this construction, we establish that $G_i$ can be cleared using only $3$ lions for an arbitrarily large $i$. Conversely, since $T_i$ contains a subdivision of a complete binary tree of height $i$, its pathwidth satisfies $\operatorname{pw}(T_i) \geq \lceil \frac{i}{2} \rceil$. According to \cref{theorem:pathwidth_theorem}, $\mathcal{L}(T_i) \geq \operatorname{pw}(T_i) \geq \lceil \frac{i}{2} \rceil$. For any $i > 6$, it follows that $\mathcal{L}(G_i) = 3 < \mathcal{L}(T_i)$. As $T_i$ is a subgraph of $G_i$, this confirms that the addition of edges or vertices to a graph can strictly decrease the required number of lions.
\end{proof}

Before proceeding to general bounds, we remark that this non-monotonicity construction has broader implications. Since the required number of lions closely bounds other pursuit-evasion variants, such as Zero-visibility Cops and Robber, we can apply this counterexample to those models as well. Due to space constraints, we detail this relationship and prove that the zero-visibility cop number is similarly non-monotonic in \cref{sec:appendix_zero}.

\section{Pathwidth Bounds on General Graphs}\label{sec:general_graph}

In this section, we formally demonstrate that $\mathcal{L}(G)$ cannot be lower-bounded by any unbounded function of pathwidth, while its upper bound strictly depends on pathwidth.

\begin{theorem}\label{theorem:pathwidth_does_not_lower_bound_lion_number_in_general_graphs}
For any positive integer $k$, there is a graph $G$ with $\mathcal{L}(G) \leq 3$ and $\operatorname{pw}(G) \geq k$.
\end{theorem}

\begin{proof}
    Consider the graph $G_i$ constructed in the proof of \cref{theorem:subgraph_counterexample}. Since this graph indeed contains a subdivision of a complete binary tree of height $i$, $\operatorname{pw}(G_i)\geq \lceil \frac{i}{2} \rceil$. However, $\mathcal{L}(G_i)\leq3$ for every integer $i\geq 1$. This implies that the lower bound of required lions on general graphs cannot be written as a function of pathwidth.
\end{proof}

Even though the lower bound of required lions on general graphs cannot be written as a function of pathwidth, there exists an upper bound on general graphs as a function of pathwidth. We use the following lemma from \cite{dereniowski2015zero}:

\begin{lemma}[Lemma 2.1 of \cite{dereniowski2015zero}]\label{lemma:proper_bag_lemma}
Let $G$ be a connected graph with $\operatorname{pw}(G) \leq |V_G|-2$. Then, there is a path decomposition $\mathcal{B}$ of $G$ containing $n \geq 2$ bags such that $\operatorname{pw}(\mathcal{B})=\operatorname{pw}(G)$ and, for each $i=1,\dots, n-1$, each of $\mathcal{B}_i \setminus \mathcal{B}_{i+1}$, $\mathcal{B}_{i+1} \setminus \mathcal{B}_i$, and $\mathcal{B}_i \cap \mathcal{B}_{i+1}$ is nonempty.
\end{lemma}

\begin{theorem}\label{theorem:upper_bound_on_general_graph_theorem}
$\mathcal{L}(G) \leq \operatorname{pw}(G)+1$.
\end{theorem}

\begin{proof}
If $|V_G|=1$, then $\mathcal{L}(G)=1\le \operatorname{pw}(G)+1=1$ always holds. If $\operatorname{pw}(G)=|V_G|-1$, we can simply place a lion on every vertex of $V_G$, immediately clearing $G$. Thus, we assume that $\operatorname{pw}(G) \leq |V_G|-2$.

Let $\mathcal{B}=(\mathcal{B}_1, \dots, \mathcal{B}_n)$ be a path decomposition of $G$ that satisfies \cref{lemma:proper_bag_lemma}, and let $U_i=(\mathcal{B}_1 \cup \dots \cup \mathcal{B}_i) \setminus \mathcal{B}_{i+1}$. Since \cref{lemma:proper_bag_lemma} ensures that each of $\mathcal{B}_i \setminus \mathcal{B}_{i+1}$ and $\mathcal{B}_{i+1} \setminus \mathcal{B}_i$ is nonempty and $|\mathcal{B}_{i}|\leq \operatorname{pw}(G)+1$ by the definition of path decomposition for all $i$, we get $|\mathcal{B}_i \cap \mathcal{B}_{i+1}| \leq \operatorname{pw}(G)$. Furthermore, suppose that $v\in U_i$ and $u\in N(U_i)$. $v\in (\mathcal{B}_1 \cup \dots \cup \mathcal{B}_i) \setminus \mathcal{B}_{i+1}$ implies that $v\notin \mathcal{B}_{i+1} \cup \dots \cup \mathcal{B}_n$, and $u\notin (\mathcal{B}_1 \cup \dots \cup \mathcal{B}_i) \setminus \mathcal{B}_{i+1}$ implies that $u\in \mathcal{B}_{i+1}$ or $u\notin (\mathcal{B}_1 \cup \dots \cup \mathcal{B}_i)$. Since $v \notin \mathcal{B}_{j}$ for $j > i$, the bag $\mathcal{B}_{i'}$ containing both $u$ and $v$ must satisfy $i' \le i$, which forces $u \in \mathcal{B}_i \cap \mathcal{B}_{i+1}$. Thus, we get $N(U_i) \subseteq \mathcal{B}_i \cap \mathcal{B}_{i+1}$.

Now, we place a lion on every vertex of $\mathcal{B}_1$. Since $|\mathcal{B}_1| \leq \operatorname{pw}(G)+1$ and $U_1 \subseteq \mathcal{B}_1$, we successfully clear $U_1$, and every vertex in the boundary $\mathcal{B}_1 \cap \mathcal{B}_2$ has a lion on it.

Assume as the inductive hypothesis that we have cleared $U_i$ and that every vertex in the boundary $\mathcal{B}_i \cap \mathcal{B}_{i+1}$ has a lion on it.
Consider the next bag $\mathcal{B}_{i+1}$. By the definition of pathwidth, we know that $|\mathcal{B}_{i+1}| \leq \operatorname{pw}(G)+1$.

Since $N(U_i) \subseteq \mathcal{B}_i \cap \mathcal{B}_{i+1}$, the lions occupying this boundary completely prevent any recontamination of $U_i$. Since $|\mathcal{B}_i \cap \mathcal{B}_{i+1}| \leq \operatorname{pw}(G)$ and lions already occupy this subset of $\mathcal{B}_{i+1}$, we have at least one remaining lion that is not on $\mathcal{B}_i \cap \mathcal{B}_{i+1}$.

Because $G$ is a connected graph, there exists a physical path in $G$ for the remaining lions to walk to the unvisited vertices of $\mathcal{B}_{i+1} \setminus \mathcal{B}_i$. Even if these lions must traverse the currently contaminated region ($V_G \setminus U_i$) taking multiple time steps, the cleared region $U_i$ remains completely safe due to the guarded boundary. Once the remaining lions reach and occupy the rest of $\mathcal{B}_{i+1}$, we successfully clear every vertex in $\mathcal{B}_{i+1}$.

Since $U_{i+1} = U_i \cup (\mathcal{B}_{i+1} \setminus \mathcal{B}_{i+2}) \subseteq U_i \cup \mathcal{B}_{i+1}$ and the new boundary $\mathcal{B}_{i+1} \cap \mathcal{B}_{i+2}$ is a subset of $\mathcal{B}_{i+1}$, this step always clears $U_{i+1}$ and places lions on every vertex in $\mathcal{B}_{i+1} \cap \mathcal{B}_{i+2}$ without recontaminating $U_i$.

Thus, by induction, $\mathcal{L}(G) \leq \operatorname{pw}(G)+1$.
\end{proof}

\begin{remark}
Since \cref{theorem:pathwidth_theorem} suggests that $\mathcal{L}(T) \leq \operatorname{pw}(T)+1$ is tight, the bound $\mathcal{L}(G) \leq \operatorname{pw}(G)+1$ is the best possible for general graphs.
\end{remark}

\section{Monotone Clearing}\label{sec:monotone_clearing}
In this section, we consider monotone clearings, which are defined as clearings where recontamination is forbidden. Formally, we call a clearing monotone if $W_{t-1}\supseteq W_{t}$ holds for every time step $t\geq 1$. Proving the upper bound of monotone clearings requires us to introduce a variant of pathwidth: the \emph{connected pathwidth}. If $\mathcal{B}=(\mathcal{B}_1, \dots, \mathcal{B}_n)$ is a path decomposition such that $\mathcal{B}_1 \cup \dots \cup \mathcal{B}_i$ is connected for every integer $1\leq i\leq n$, then the connected pathwidth of the graph $G$ is the minimum of $\max_{i=1,\dots,n}|\mathcal{B}_i| - 1$ over all such decompositions. From the following theorem from \cite{dereniowski2012pathwidth}, we now prove the upper bound on $\mathcal{L}^m(G)$:

\begin{lemma}[Lemma 15 of \cite{dereniowski2012pathwidth}]\label{lemma:connected_pathwidth_lemma}
For each connected graph $G$, $\operatorname{cpw}(G)\leq 2\cdot \operatorname{pw}(G) + 1$.
\end{lemma}

\begin{theorem}\label{theorem:monotonous_upper_bound_on_general_graph_theorem}
$\mathcal{L}^m(G) \leq 2\operatorname{pw}(G)+2$.
\end{theorem}

\begin{proof}
Since \cref{lemma:connected_pathwidth_lemma} suggests that connected pathwidth $\operatorname{cpw}(G) \leq 2\operatorname{pw}(G)+1$, we can follow the strategy from \cref{theorem:upper_bound_on_general_graph_theorem}. Since no vertex in $\mathcal{B}_{i+1}$ is ever recontaminated while clearing $\mathcal{B}_{i+1}$, recontamination occurs only when the lions are moving to fill the new vertices of $\mathcal{B}_{i+1}$. However, the connected pathwidth assures us that the subgraph induced by $\mathcal{B}_1 \cup \dots \cup \mathcal{B}_{i+1}$ is connected. This guarantees that for the lions on $\mathcal{B}_i$, there always exists a valid physical path to move them to the frontier $\mathcal{B}_i \cap \mathcal{B}_{i+1}$ and subsequently spread into the rest of $\mathcal{B}_{i+1}$ without leaving the cleared subgraph. Thus, we can completely prevent recontamination at every step, allowing a monotonic clearing path.
\end{proof}

We can show that this bound is tight up to a small additive constant. Due to space constraints, the proof is provided in \cref{sec:appendixproof}. We now prove that $\mathcal{L}^m(G)$ has a lower bound in terms of pathwidth. To prove this, we first establish the following lemma:

\begin{lemma}\label{lemma:boundary_keeping_lemma}
Let $\pi=\{\pi_1, \dots, \pi_n\}$ be a monotonic path to clear the graph $G$ using $k$ lions. Then, $\partial(C_t) \subseteq L_t$ for every $t \geq 0$.
\end{lemma}

\begin{proof}
We prove this by contradiction. Suppose there exists a time step $t$ and a vertex $v$ such that $v \in \partial(C_t)$ but $v \notin L_t$.
Since $v \in \partial(C_t)$, we know that $v \in C_t$ and there exists a neighbor $u \in W_t \cap N(v)$. Because $\pi$ is a monotonic path ($W_t \subseteq W_{t-1}$), it must be that $u \in W_{t-1}$.
Consider the state of vertex $v$ at time step $t-1$:

\begin{enumerate}
    \item If $v \in W_{t-1}$: By the problem definition, since $v\in C_t$, it must hold that $v\in L_t$. However, this contradicts our assumption that $v \notin L_t$.
    \item If $v \in C_{t-1}$: Since $u \in W_{t-1}$, $v$ has a contaminated neighbor at time $t-1$. For $v$ to remain in $C_t$ at time $t$ while $v\notin L_t$, a lion must traverse the edge between $v$ and $u$ to block the contamination. Since $v \notin L_t$, the lion cannot move from $u$ to $v$; it must move from $v$ to $u$. However, if a lion arrives at $u$ at time $t$, then $u \in L_t \subseteq C_t$, which contradicts the fact that $u \in W_t$.
\end{enumerate}

Both cases lead to a contradiction. Thus, every vertex in the boundary $\partial(C_t)$ must be occupied by a lion at time $t$.
\end{proof}

Now, we use following theorem from \cite{Bertschinger23} to prove the lower bound:
\begin{theorem}[Theorem 11 of \cite{Bertschinger23}] \label{theorem:monotonous_polite_theorem}
Let $G$ be a graph and $C$ be a monotone clearing of $G$ with $k$ lions. Then, there exists another monotone clearing which uses $k$ polite lions.
\end{theorem}
Here, we say the lions are \emph{polite} if at most one lion moves at each time step.

\begin{theorem}\label{theorem:monotonous_lower_bound_on_general_graph_theorem}
$\operatorname{pw}(G) \leq \mathcal{L}^m(G)$.
\end{theorem}

\begin{proof}
Let $\pi=\{\pi_1, \dots, \pi_n\}$ be a monotonic path to clear the graph $G$ with $k$ polite lions where $\mathcal{L}^m(G)=k$. Note that, by \cref{theorem:monotonous_polite_theorem}, we always have monotonic path $\pi$ with $k$ polite lions. By the definition of monotonicity, we have $W_{t} \subseteq W_{t-1}$ for every integer $1 \leq t \leq n$.
For the path decomposition $\mathcal{B}$, let $\mathcal{B}_0 = L_0$ and $\mathcal{B}_t = \partial(C_{t-1}) \cup (L_t \cap W_{t-1})$ for every integer $1 \leq t \leq n$.
Since $\pi$ is a monotonic path, by \cref{lemma:boundary_keeping_lemma} we get $\partial(C_{t-1}) \subseteq L_{t-1}$, and $|L_{t-1}\cup L_t|\leq k+1$ since at most one lion is allowed to move at each time step. Therefore, we get the bound $|\mathcal{B}_t| \leq |(\partial(C_{t-1})\cup L_t) \cap W_{t-1}| \leq |L_{t-1}\cup L_t| \leq k+1$.
Now, we will use conditions in \cref{prop:pw-basics} to show that $\mathcal{B}$ is a valid path decomposition of $G$.

\proofsubparagraph*{Condition (i)} Since $\pi$ successfully clears the graph $G$, every initially contaminated vertex in $V_G$ must be visited by a lion at least once to be cleared. Thus, $\bigcup_{i=0}^n \mathcal{B}_i \supseteq \bigcup_{i=1}^n (L_i \cap W_{i-1}) \cup L_0 = V_G$, satisfying the first condition.

\proofsubparagraph*{Condition (ii)} Let $v \in \mathcal{B}_t$, $v \notin \mathcal{B}_{t+1}$, and $v\in \mathcal{B}_{t+k}$ for some $k>1$ and $t+k\leq n$. Since $v \in \mathcal{B}_t$, it follows that $v$ is cleared at or before time $t$, meaning $v \in C_t \subseteq C_{t+1} \subseteq \dots \subseteq C_n$. Then, since $v \notin \mathcal{B}_{t+1}$, we know $v \notin \partial(C_t)$. Because $\pi$ is a monotonic clearing path, $v \notin \partial(C_i)$ holds for $t\leq i\leq n$. Therefore, for $v\in \mathcal{B}_{t+k}$, $v \in L_{t+k} \cap W_{t+k-1}$ would need to hold. However, since $v \in C_t \subseteq C_{t+k-1}$, $v \notin W_{t+k-1}$. Therefore, $v\notin L_{t+k} \cap W_{t+k-1}$, which implies $v\notin \mathcal{B}_{t+k}$, thereby satisfying the second condition.

\proofsubparagraph*{Condition (iii)} Suppose $v \in L_t \cap W_{t-1} \subseteq \mathcal{B}_t$ (i.e., $v$ is newly cleared at time $t$). Note that every vertex in $V_G \setminus L_0$ satisfies this condition for some $t$.
Consider the neighborhood $N(v)$. By the definition of the boundary, every vertex in $N(v)$ must belong to either $\partial(C_{t-1})$ or $W_{t-1}$.
Since $v \in L_t \cap W_{t-1} \subseteq \mathcal{B}_t$ and $\partial(C_{t-1}) \subseteq \mathcal{B}_t$, every vertex in $N(v) \cap \partial(C_{t-1})$ is in $\mathcal{B}_t$. Therefore, the edge between $v$ and any already-cleared neighbor is covered in $\mathcal{B}_t$.
If $N(v) \cap W_{t-1} \neq \emptyset$, let $u \in N(v)\cap W_{t-1}$ be such a neighbor. If $u\in L_t\cap W_{t-1}$, then $u\in \mathcal{B}_t$. Otherwise, $u$ remains contaminated while $v$ is cleared, $v$ becomes part of the boundary $\partial(C_t)$ and remains in $\partial(C_i)$ for every time step $t \leq i < t'$, where $t'$ is the step at which $u$ is finally cleared ($u \in W_{t'-1} \cap L_{t'}$). Consequently, $v \in \partial(C_{t'-1}) \subseteq \mathcal{B}_{t'}$ and $u \in L_{t'} \cap W_{t'-1} \subseteq \mathcal{B}_{t'}$. This implies that for every vertex in $V_G \setminus L_0$, every incident edge is covered in some bag $\mathcal{B}_i$.
Finally, if $v \in L_0 \setminus \partial(C_0)$, then $N(v) \subseteq L_0$, meaning both $v$ and $N(v)$ are entirely in $\mathcal{B}_0$. If $v \in L_0 \cap \partial(C_0)$, there exists a neighbor $u \in W_0$ which is cleared at some later time $t$. By the same logic as above, $v$ remains in $\partial(C_{t-1})$, so both $v$ and $u$ appear together in $\mathcal{B}_t$. Thus, the third condition holds.

Since all three conditions for a valid path decomposition hold, $\mathcal{B}$ is a valid path decomposition of $G$. Since $\max_i |\mathcal{B}_i| \leq k+1$, we conclude that $\operatorname{pw}(G) \leq \max_i |\mathcal{B}_i| - 1 \leq k = \mathcal{L}^m(G)$.
\end{proof}

Note that for a complete graph $K_m$, it is easy to show that $\operatorname{pw}(K_m)=m-1=\mathcal{L}^m(K_m)$ for every $m$, thus the general lower bound is strictly tight. Furthermore, this result gives us $\operatorname{pw}(G) \leq \mathcal{L}^m(G) \leq 2\operatorname{pw}(G)+2$, resolving Open Question 1 posed in \cite{Bertschinger23}:

\begin{quote}
\emph{Open Question 1.}\cite{Bertschinger23} Given a $k$-clearable graph $G = (V, E)$, is it always monotonically clearable with $k + 1$ lions? More formally, is there a non-trivial upper bound on the number of lions required for a monotone clearing?
\end{quote}

From \cref{theorem:pathwidth_does_not_lower_bound_lion_number_in_general_graphs}, we can construct a graph $G$ with $\mathcal{L}(G) \leq 3$ and an arbitrarily large pathwidth $\operatorname{pw}(G) \geq K$ for any positive integer $K$. However, since the lower bound $\operatorname{pw}(G) \leq \mathcal{L}^m(G)$ always holds, $\mathcal{L}^m(G)$ can scale infinitely with $\operatorname{pw}(G)$. Thus, even if a graph $G$ is $k$-clearable, it is not necessarily monotonically clearable with $k+1$ lions. The gap between $\mathcal{L}(G)$ and $\mathcal{L}^m(G)$ is unbounded.
Furthermore, we provide a non-trivial upper bound on the number of lions required for a monotone clearing as $\mathcal{L}^m(G) \leq 2\operatorname{pw}(G)+2$.

\section{Conclusion}

We investigated the lions and contamination problem through the lens of pathwidth, examining both the ordinary lion number $\mathcal{L}(G)$ and the monotone lion number $\mathcal{L}^m(G)$.
In \cref{sec:tree}, we proved an isometric-subgraph monotonicity theorem, showing that \(\mathcal{L}(H)\le \mathcal{L}(G)\) whenever \(H\) is an isometric subgraph of \(G\). Using this property, we established that \(\operatorname{pw}(T)\le \mathcal{L}(T)\le \operatorname{pw}(T)+1\). 
Since both extremes are achievable, this provides the best possible characterization of the lion number strictly in terms of pathwidth on trees.
In \cref{sec:subgraph_counterexample}, we provided a counterexample showing that monotonicity fails for general subgraphs. This resolves the subgraph question from previous work in the negative.
In \cref{sec:general_graph}, we extended the upper bound to all connected graphs by proving $\mathcal{L}(G)\le \operatorname{pw}(G)+1$.
We also showed that pathwidth does not provide a general lower bound for \(\mathcal{L}(G)\).
This establishes a fundamental separation between the non-monotone lions and contamination model and the behavioral paradigms of classical graph searching.
Additionally, as an extension of our counterexample, we note that the zero-visibility cop number is similarly non-monotonic with respect to subgraph inclusion.
In \cref{sec:monotone_clearing}, we established lower and upper bounds relating \(\mathcal{L}^m(G)\) and \(\operatorname{pw}(G)\), demonstrating a constant-factor equivalence between the monotone lion number and pathwidth on connected graphs. Importantly, this completely resolves Open Question 1 posed by Bertschinger et al.~\cite{Bertschinger23} regarding the gap between ordinary and monotone clearings, proving that $\mathcal{L}^m(G)$ can indeed scale infinitely independent of $\mathcal{L}(G)$. Our results provide a comprehensive pathwidth-based framework for lions and contamination on general graphs.

Several questions remain open. It is natural to seek sharper characterizations of \(\mathcal{L}(G)\) beyond trees, for example on restricted graph classes such as planar graphs or bounded-treewidth graphs, and to better understand which graph-containment notions preserve or approximate the lion number. Another natural direction is to determine stronger lower bounds for the lion number on \(n\times n\) grids, especially in forms that may connect grid geometry with width parameters.



\bibliography{lipics-v2021-sample-article}

\appendix

\section{Connection to Zero-visibility Cops and Robber}\label{sec:appendix_zero}
In this section, we introduce the relationship between the lions and contamination problem and Zero-visibility Cops and Robber from \cite{dereniowski2015zero}.
This game can be viewed, from the perspective of lions and contamination, as a non-simultaneous variant where the cops and robber move sequentially. We call a vertex \emph{dirty} if the vertex might have a robber on it, which is analogous to contamination in our problem.

More formally, if we denote $P_t$ as the set of vertices occupied by cops at the end of the $t$-th turn and if we denote $S_t$ and $R_t$ as the sets of vertices that are dirty immediately after and before the cop's $t$-th turn, respectively, then $S_t=R_t\setminus P_t$, $R_{t+1}=(N(S_t)\cup S_t)\setminus P_t$, and each cop in $P_{t+1}$ moves to an adjacent vertex or stays at their current vertex in $P_t$, satisfying $P_{t+1}\subseteq N(P_t)\cup P_t$. Let $c_0(G)$ denote the zero-visibility cop number, which is the minimum number of cops required to clear the graph $G$.

\begin{lemma}\label{lemma:cop_lower_bound}
$c_0(G) \le \mathcal{L}(G)$
\end{lemma}

\begin{proof}
By the definition of the lions and contamination problem, the contaminated set $W_{t+1}$ consists of all vertices $v \notin L_{t+1}$ such that $v \in W_t \setminus L_t$, or there exists a contaminated neighbor $w \in W_t$ where the edge $\{w, v\}$ is not traversed by a lion at time $t+1$.
On the other hand, under the standard Zero-visibility Cops and Robber rules, the set of dirty vertices $S_{t+1}$ consists of all vertices $v \notin \pi^c_{t+1}$ (where $\pi^c_{t+1}$ is the set of cop positions) such that $v \in S_t \setminus \pi^c_t$, or there is a dirty neighbor $w \in S_t$ that is not occupied by a cop.
Comparing the definitions, the condition for a vertex to become (or remain) \emph{dirty} is equivalent to or strictly looser than the condition to become contaminated, since lions must actively block edges during movement to stop contamination spread, whereas cops simply occupying vertices restrict the robber. Thus, assuming identical movement paths for cops and lions, a dirty vertex under the Cops rules is always contaminated under the Lions rules. More formally, $S_t \subseteq W_t$. Since a valid clearing strategy requires the respective sets to become empty, if a lion strategy clears $W_t = \emptyset$, the identical cop strategy guarantees $S_t = \emptyset$. It follows that $c_0(G) \le \mathcal{L}(G)$.
\end{proof}

\begin{lemma}
$\mathcal{L}(G) \le 2c_0(G)$.
\end{lemma}

\begin{proof}
Let $\pi^c$ be a cop path that clears the graph $G$ using $c_0(G)$ cops. For each cop $c$ following a trajectory of vertices $(v_0, v_1, v_2, \dots)$ across time steps, we assign exactly two lions and we define their moves at time step $t$ as $(v_{t-2}, v_{t-1})$ and $(v_{t-1}, v_t)$ where $v_{-1}=v_0$. Note that at time step $t$, the two lions assigned to this cop exactly occupy $v_{t-1}$ and $v_t$, which naturally ensures that $P_{t-1}\cup P_{t}=L_t$ holds. Since $P_{t+1}\subseteq N(P_t)\cup P_t$, these are physically valid lion moves because they mimic valid cop movements along adjacent vertices.

Let time step $t$ be the first time step where there exists a vertex $v \in W_t$ under $\pi$ but $v\notin S_t$ under $\pi^c$. By the definition of the problem, $W_0\subseteq S_0$ holds. By definition, $S_t=R_t\setminus P_t$, and since $v\notin P_t$, we require that $v\notin R_t=(N(S_{t-1})\cup S_{t-1})\setminus P_{t-1}$ holds for $v\notin S_t$, which implies $v\in P_{t-1}$ or $v\notin N(S_{t-1})\cup S_{t-1}$. If $v\in P_{t-1}$, then $P_{t-1}\subseteq L_t$ implies that $v\notin W_t$, which is a contradiction. Thus, $v\notin N(S_{t-1})\cup S_{t-1}$. Furthermore, since $t$ is the \emph{first} time step where a vertex is in $W_t$ but not in $S_t$ and $v\notin S_{t-1}$, $v\notin W_{t-1}$ also holds.

For $v$ to be in $W_t$, there must exist a vertex $u$ such that $u\in W_{t-1}\cap N(v)$ (since $v\notin W_{t-1}$, $v\neq u$). Similarly, by $t$ being the \emph{first} time step, $u\in S_{t-1}$ and $u\in N(v)$ also holds. Since $u\in N(v)$ implies that $v\in N(u)$, $v\in N(u)\subseteq N(S_{t-1})$, which contradicts $v\notin N(S_{t-1})\cup S_{t-1}$. Since $v\in W_t$ and $v\notin S_t$ cannot happen at the same time, $W_t\subseteq S_t$ holds for every $t$. Since $\pi^c$ clears $G$, and thus $S_t=\emptyset$ for sufficiently large $t$, $\pi$ can also clear $G$ with $2c_0(G)$ lions.
\end{proof}

\begin{corollary}\label{corollary:cop_tree_bound}
    For a tree $T$ that contains two or more vertices, $\lceil{\frac{\operatorname{pw}(T)}{2}}\rceil \leq c_0(T) \leq \operatorname{pw}(T)$.
\end{corollary}
\begin{proof}
Since $\mathcal{L}(T) \le 2c_0(T)$ and $\operatorname{pw}(T)\le \mathcal{L}(T)$ by \cref{theorem:pathwidth_theorem}, we get $\lceil{\frac{\operatorname{pw}(T)}{2}}\rceil \leq c_0(T)$.
For the upper bound, we refer to Theorem 3.2 from \cite{dereniowski2015zero}:
\begin{quote}
Theorem 3.2 $c_0(G) \leq \operatorname{pw}(G)$ for any connected graph containing two or more vertices.
\end{quote}
Combining these two inequalities, we obtain $\lceil{\frac{\operatorname{pw}(T)}{2}}\rceil \leq c_0(T) \leq \operatorname{pw}(T)$.
\end{proof}
Furthermore, by applying the graph construction from \cref{theorem:subgraph_counterexample}, we obtain another important consequence regarding the zero-visibility cop number.
\begin{corollary}
    The zero-visibility cop number $c_0(G)$ is not monotonic with respect to subgraph inclusion.
\end{corollary}
\begin{proof}
    Consider the graphs $T_i$ and $G_i$ constructed in the proof of \cref{theorem:subgraph_counterexample}. We know $\mathcal{L}(G_i)\leq3$ for any $i$.
    By \cref{lemma:cop_lower_bound}, $c_0(G_i) \le \mathcal{L}(G_i)$, which implies $c_0(G_i) \le 3$.
    Furthermore, the lower bound $\lceil{\frac{\operatorname{pw}(T_i)}{2}}\rceil \leq c_0(T_i)$ from \cref{corollary:cop_tree_bound} ensures that $c_0(T_i) \ge \lceil \frac{1}{2} \lceil \frac{i}{2} \rceil \rceil$.
    For sufficiently large $i$ (e.g., $i \ge 13$), we get $c_0(T_i) \ge 4 > 3 \ge c_0(G_i)$.
    Since $T_i$ is a subgraph of $G_i$, this definitively demonstrates that a subgraph might require strictly more zero-visibility cops than its supergraph.
\end{proof}

\section{Proof of Near-Tightness on the bound in \cref{theorem:monotonous_upper_bound_on_general_graph_theorem}}\label{sec:appendixproof}
In this section, we provide the proof that the bound given in \cref{theorem:monotonous_upper_bound_on_general_graph_theorem} is tight up to a constant term by showing that $\mathcal{L}^m(T_h) = h$ for a complete binary tree $T_h$ of height $h$ (i.e., a complete binary tree with $2^{h+1}-1$ vertices) for every integer $h \ge 1$.

\begin{lemma}
    For a complete binary tree $T_h$ with $2^{h+1}-1$ vertices, $\mathcal{L}^m(T_h)=h$ for every integer $h \ge 1$.
\end{lemma}
\begin{proof}
First, we establish a bound for a restricted scenario: let $\pi$ be a monotone clearing path for $T_h$ where all lions initially start at the root $r$. We claim that $\mathcal{L}_\pi(T_h) \ge h+1$.
We prove this by induction on $h$. For the base case $h=1$, it is trivial that $\mathcal{L}_\pi(T_1) \ge 2 = 1+1$.
Assume the claim holds for $T_{h-1}$. When clearing $T_h$ monotonically via $\pi$ starting from $r$, the lions must eventually clear its two maximal subtrees, both isomorphic to $T_{h-1}$. By \cref{lemma:subtree_on-graph_lion_lemma}, clearing one such subtree $T'$ requires the lions to fully enter $T'$ from $r$. By the inductive hypothesis, clearing $T'$ internally under this restriction requires at least $h$ lions. Meanwhile, to maintain monotonicity and prevent contamination from the other uncleared subtree from spreading through $r$ into $T'$, at least one additional lion must guard the root $r$ by \cref{lemma:boundary_keeping_lemma}. Therefore, $\mathcal{L}_\pi(T_h) \ge h+1$.

Now, we prove the main claim $\mathcal{L}^m(T_h)= h$ by induction on $h$.
For the base cases $h=1$ and $h=2$, it is easy to verify that $\mathcal{L}^m(T_1)=1$ and $\mathcal{L}^m(T_2)=2$.
Assume as our inductive hypothesis that $\mathcal{L}^m(T_k)=k$ for some integer $k \ge 2$.
We will show that $\mathcal{L}^m(T_{k+1}) = k+1$.
Assume for contradiction that $\mathcal{L}^m(T_{k+1}) \le k$, and let $\pi^*$ be a monotone path clearing $T_{k+1}$ with at most $k$ lions.
Let $r$ be the root of $T_{k+1}$, and let $T_L$ and $T_R$ be its two maximal subtrees.
By the inductive hypothesis, $\mathcal{L}^m(T_L) = \mathcal{L}^m(T_R) = k$.
By applying the remote clear logic from \cref{lemma:subtree_on-graph_lion_lemma} to monotone paths, any monotone clearing must place at least $\mathcal{L}^m(T_L) = k$ lions entirely within $T_L$ at some time step $t_L$. Since we assumed $\pi^*$ uses at most $k$ lions, all available lions must be located entirely within $T_L$ at time $t_L$. Similarly, there exists a time step $t_R$ for $T_R$. Assume without loss of generality that $t_L < t_R$.
At time $t_L$, no lions are present in the connected subgraph $T_R \cup \{r\}$. 
By \cref{lemma:boundary_keeping_lemma}, the absence of lions and $t_L<t_R$ implies that $T_R\cup \{r\}$ must be fully contaminated at time $t_L$.
Subsequently clearing $T_R$ monotonically from this state requires the lions to enter via $r$. This exactly reduces to the restricted root-starting scenario for a tree of height $k$, demanding at least $k+1$ lions. This contradicts our assumption that $\pi^*$ uses at most $k$ lions, so $\mathcal{L}^m(T_{k+1}) \ge k+1$.
Finally, stationing one lion at $r$ while using the remaining $k$ lions to sequentially clear $T_L$ and $T_R$ is a valid monotone strategy, yielding $\mathcal{L}^m(T_{k+1}) \le k+1$.
Thus, $\mathcal{L}^m(T_{k+1}) = k+1$.
\end{proof}
With this lemma established, we can now demonstrate the near-tightness of the upper bound given in \cref{theorem:monotonous_upper_bound_on_general_graph_theorem}. 
It is a known result that the pathwidth of a complete binary tree $T_h$ of height $h$ is $\operatorname{pw}(T_h) = \lceil \frac{h}{2} \rceil$ \cite{bodlaender1998partial}. 
Since we proved $\mathcal{L}^m(T_h) = h$, we can express the monotone lion number in terms of pathwidth as $2\operatorname{pw}(T_h) - 1 \le \mathcal{L}^m(T_h) \le 2\operatorname{pw}(T_h)$. 
Recall that the general upper bound from \cref{theorem:monotonous_upper_bound_on_general_graph_theorem} states $\mathcal{L}^m(G) \le 2\operatorname{pw}(G) + 2$ for any connected graph $G$. 
Comparing these results, the minimum number of lions required for $T_h$ matches the general upper bound up to a difference of at most $3$. 
This confirms that the multiplicative factor of $2$ in our general upper bound is essential and cannot be improved, establishing its near-tightness.
\end{document}